\documentclass[reqno,12pt]{article}
\usepackage{amsfonts,amssymb,amsmath,amsthm,epsfig,mathrsfs}
\usepackage{xcolor}

\newcommand{\R}{\mathbb{R}}

\newcommand{\res}{\mathop{\hbox{\vrule height 7pt width .5pt depth 0pt
\vrule height .5pt width 6pt depth 0pt}}\nolimits}

\newcommand{\Haus}[1]{{\mathscr S}^{#1}} 
\newcommand{\Leb}[1]{{\mathscr L}^{#1}} 

\newcommand{\eps}{\varepsilon}

\newtheorem{theorem}{Theorem}[section]
\newtheorem{lemma}[theorem]{Lemma}
\newtheorem{definition}[theorem]{Definition}
\newtheorem{remark}[theorem]{Remark}
\newtheorem{proposition}[theorem]{Proposition}

\newtheorem{ex}[theorem]{Example}

\title{On sets of finite perimeter in Wiener spaces:\\
reduced boundary and convergence to halfspaces}
\author{Luigi Ambrosio\footnote{Scuola Normale Superiore,
p.za dei Cavalieri 7, I-56126 Pisa, Italy.  e--mail:
l.ambrosio@sns.it}, Alessio Figalli\footnote{The University of Texas
at Austin, Department of Mathematics, RLM 8.100,
2515 Speedway Stop C1200,
Austin TX 78712, USA. e--mail: figalli@math.utexas.edu}, Eris
Runa\footnote{Hausdorff Center for Mathematics, Universit\"at Bonn,
Germany. e--mail: eris.runa@hcm.uni-bonn.de}}

\usepackage{esint}

\begin{document}

\maketitle

\begin{abstract}
We study sets of finite perimeter in Wiener space, and prove that at
almost every point (with respect to the perimeter measure) a set of finite
perimeter blows-up to a halfspace.
\end{abstract}

\section{Introduction}

The theory of sets of finite perimeter and $BV$ functions in Wiener
spaces, i.e., Banach spaces endowed with a Gaussian Borel
probability measure $\gamma$, was initiated by Fukushima and
Hino in \cite{fuk99,fuk2000_1,fuk2000_2}, and has been further
investigated in \cite{hin09set,AMMP,AMP,ambfig10}.

The basic question one would like to consider is the research of
infinite-dimensional analogues of the classical fine properties of
$BV$ functions and sets of finite perimeter in finite-dimensional
spaces. The class of sets of finite Gaussian perimeter $E$ in a Gaussian Banach space
$(X,\gamma)$ is defined by the integration by parts formula
$$
\int_E\partial_h\phi\,d\gamma=-\int_X \phi\,d\langle D_\gamma\chi_E,h\rangle_H+\int_E\phi\hat{h}\,d\gamma
$$
for all $\phi\in C^1_b(X)$ and $h\in H$. Here $H$ is the Cameron-Martin space of $(X,\gamma)$ and $D_\gamma\chi_E$ is
a $H$-valued measure with finite total variation in $X$. 

When looking for the counterpart of De~Giorgi's and Federer's classical
results to infinite-dimensional spaces, it was noticed in \cite{ambfig10}
that the Ornstein-Uhlenbeck 
$$T_t\chi_E(x):=\int_X\chi_E(e^{-t}x+\sqrt{1-e^{-2t}}y)\,d\gamma(y)$$ 
can be used to rephrase the notion of density, the main result of that paper being 
\begin{equation}\label{eq:ambfig10}
\lim_{t\downarrow 0}\int_X\Bigl|T_t\chi_E-\frac{1}{2}\Bigr|\,d|D_\gamma\chi_E|=0.
\end{equation}
According to this formula, we might say that $|D_\gamma\chi_E|$ is concentrated
on the set of points of density $1/2$, where the latter set is not defined using volume ratio in balls (as in
the finite-dimensional theory), but rather the Ornstein-Uhlenbeck semigroup.

In this paper we improve \eqref{eq:ambfig10} as follows (we refer to
Section~\ref{sec:halfspaces} for the notation relative to halfspaces):

\begin{theorem} \label{main} Let $E$ be a set of finite perimeter in $(X,\gamma)$ and let $S(x)=S_{\nu_E(x)}$ be the halfspaces determined 
by $\nu_E(x)$. Then
\begin{equation}\label{eq:main}
\lim_{t\downarrow 0}\int_X\int_X\left|\chi_E(e^{-t}x+\sqrt{1-e^{-2t}}y)-\chi_{S(x)}(y)\right|\,d\gamma(y)\,d|D_\gamma\chi_E|(x)=0.
\end{equation}
\end{theorem}

A nice interpretation of this result can be obtained stating it in terms of the Gaussian rescaled sets
$$
E_{x,t}=\frac{E-e^{-t}x}{\sqrt{1-e^{-2t}}},
$$
namely
\begin{equation}\label{eq:main1}
\lim_{t\downarrow 0}\int_X\|\chi_{E_{x,t}}-\chi_{S(x)}\|_{L^1(\gamma)}\,d|D_\gamma\chi_E|(x)=0.
\end{equation}
Clearly, if we pull the modulus out of the integral in \eqref{eq:main} we recover \eqref{eq:ambfig10}, because the measure
of halfspaces is $1/2$ and $T_t\chi_E(x)=\gamma(E_{x,t})$. More specifically, \eqref{eq:main1} formalizes the fact, established by De~Giorgi
in finite dimensions, that on small scales a set of finite perimeter is close to an halfspace
  at almost every (w.r.t. surface measure).

The proof of \eqref{eq:main1} relies mainly on a combination of the careful finite-dimensional estimates
of \cite{ambfig10} with a variant
of the cylindrical construction performed in \cite{hin09set} (with
respect to \cite{hin09set}, here we use of the reduced boundary instead
of the essential boundary of the
finite-dimensional sections of $E$).

\section{Preliminary results}

We assume that $(X,\|\cdot\|)$ is a separable Banach space and
$\gamma$ is a Gaussian probability measure on the Borel
$\sigma$-algebra of $X$. We shall always assume that $\gamma$ is
nondegenerate (i.e., all closed proper subspaces of $X$ are
$\gamma$-negligible) and centered (i.e., $\int_X x\,d\gamma=0$). We
denote by $H$ the Cameron-Martin subspace of $X$, that is
$$
H:=\left\{\int_X f(x)x\,d\gamma(x):f\in L^2(X,\gamma)\right\},
$$
and, for $h \in H$, we denote by $\hat h \in L^2(X,\gamma)$ the Fomin
derivative of $\gamma$ along $h$, namely
\begin{equation}\label{Fomin}
\int_X\partial_h\phi\,d\gamma=-\int_X\hat{h}\phi\,d\gamma
\end{equation}
for all $\phi\in C^1_b(X)$. Here and in the sequel $C^1_b(X)$
denotes the space of continuously differentiable cylindrical
functions in $X$, bounded and with a bounded gradient. The space $H$
can be endowed with a Hilbertian norm $|\cdot |_H$ that makes the
map $h\mapsto\hat{h}$ an isometry; furthermore, the injection of
$(H,|\cdot |_H)$ into $(X,\|\cdot\|)$ is compact.

We shall denote by $\tilde{H}\subset H$ the subset of vectors of the form
\begin{equation}\label{defhstar}
\int_X \langle x^*,x\rangle x\,d\gamma(x),\qquad x^*\in X^*.
\end{equation}
This is a dense (even w.r.t. to the Hilbertian norm) subspace of
$H$. Furthermore, for $h\in \tilde H$ the function $\hat{h}(x)$ is
precisely $\langle x^*,x\rangle$ (and so, it is continuous).

Given an $m$-dimensional subspace $F\subset \tilde{H}$ we shall frequently
consider an orthonormal basis $\{h_1,\ldots,h_m\}$ of $F$ and the
factorization $X=F\oplus Y$, where $Y$ is the kernel of the
continuous linear map
\begin{equation}\label{ammiss1}
x\in X\mapsto \Pi_F(x):=\sum_{i=1}^m\hat{h}_i(x)h_i\in F.
\end{equation}
The decomposition $x=\Pi_F(x)+(x-\Pi_F(x))$ is well defined, thanks
to the fact that $\Pi_F\circ \Pi_F=\Pi_F$ and so $x-\Pi_F(x)\in Y$;
in turn this follows by
$\hat{h}_i(h_j)=\langle\hat{h}_i,\hat{h}_j\rangle_{L^2}=\delta_{ij}$.

Thanks to the fact that $|h_i|_H=1$, this induces a factorization
$\gamma=\gamma_F\otimes\gamma_Y$, with $\gamma_F$ the standard Gaussian
in $F$ (endowed with the metric inherited from $H$) and $\gamma_Y$
Gaussian in $(Y,\|\cdot\|)$. Furthermore, the orthogonal complement
$F^\perp$ of $F$ in $H$ is the Cameron-Martin space of
$(Y,\gamma_Y)$.

\subsection{$BV$ functions and Sobolev spaces}

Here we present the definitions of Sobolev and $BV$ spaces. Since we
will consider bounded functions only, we shall restrict to this
class for ease of exposition.

Let $u:X\to\R$ be a bounded Borel function. Motivated by
\eqref{Fomin}, we say that $u\in W^{1,1}(X,\gamma)$ if there exists
a (unique) $H$-valued function, denoted by $\nabla u$, such that $|\nabla
u|_H\in L^1(X,\gamma)$ and
$$
\int_X u\partial_h\phi\,d\gamma=-\int_X \phi\langle\nabla
u,h\rangle_H\,d\gamma+\int_X u\phi\hat{h}\,d\gamma
$$
for all $\phi\in C^1_b(X)$ and $h\in H$.

Analogously, following \cite{fuk2000_1,fuk2000_2}, we say that $u\in
BV(X,\gamma)$ if there exists a (unique) $H$-valued Borel measure
$D_\gamma u$ with finite total variation in $X$ satisfying
$$
\int_X u\partial_h\phi\,d\gamma=-\int_X \phi\,d\langle D_\gamma
u,h\rangle_H+\int_X u\phi\hat{h}\,d\gamma
$$
for all $\phi\in C^1_b(X)$ and $h\in H$.

In the sequel we will mostly consider
the case when $u=\chi_E:X\to\{0,1\}$ is the characteristic function
of a set $E$, although some statements are more natural in the
general $BV$ context. Notice the inclusion $W^{1,1}(X,\gamma)\subset
BV(X,\gamma)$, given by the identity $D_\gamma u=\nabla u\,\gamma$.

\subsection{The OU semigroup and Mehler's formula}

In this paper, the Ornstein-Uhlenbeck semigroup $T_tf$ will always
be understood as defined by the \emph{pointwise} formula
\begin{equation}\label{mehler}
T_tf(x):=\int_X f(e^{-t}x+\sqrt{1-e^{-2t}}y)\,d\gamma(y)
\end{equation}
which makes sense whenever $f$ is bounded and Borel. This convention
will be important when integrating $T_t f$ against potentially
singular measures. 

We shall also use the dual OU semigroup $T_t^*$, mapping signed
measures into signed measures, defined by the formula
\begin{equation}
\langle T_t^*\mu,\phi\rangle:=\int_X T_t\phi\,d\mu\qquad\text{$\phi$
bounded Borel.}
\end{equation}

In the next proposition we collect a few properties of the OU
semigroup needed in the sequel (see for instance \cite{boga} for the
Sobolev case, and \cite{AMP} for the $BV$ case).

\begin{proposition}\label{pammiss1} Let $u:X\to\R$ be bounded and Borel, and $t>0$.
Then $T_tu\in W^{1,1}(X,\gamma)$ and:
\begin{itemize}
\item[(a)] if $u\in W^{1,1}(X,\gamma)$ then, componentwise, it holds $\nabla T_tu=e^{-t}T_t\nabla u$;
\item[(b)] if $u\in BV(X,\gamma)$ then, componentwise, it holds $\nabla T_tu\, \gamma=e^{-t}T_t^*(D_\gamma
u)$.
\end{itemize}
\end{proposition}

The next result is basically contained in
\cite[Proposition~5.4.8]{boga}, see also
\cite[Proposition~2.2]{ambfig10} for a detailed proof.
We state it in order to emphasize that, $\gamma_Y$-a.e. $y \in Y$,
the regular version of the restriction of
$T_tf$ to $y+F$ (provided by the above proposition) is for
precisely the one pointwise defined in Mehler's
formula.

\begin{proposition} \label{pbogaregu} Let $u$ be a bounded Borel function and $t>0$.
With the above notation, for $\gamma_Y$-a.e. $y\in Y$ the map
$z\mapsto T_t u(z,y)$ is smooth in $F$.
\end{proposition}

The next lemma provides a rate of convergence of $T_t u$ to $u$ when
$u$ belongs to $BV(X,\gamma)$; the proof follows the lines of the
proof of Poincar\'e inequalities, see \cite[Lemma~2.3]{ambfig10},
\cite[Theorem~5.5.11]{boga}.

\begin{lemma}\label{lpoincare} Let $u\in BV(X,\gamma)$. Then
\begin{equation}\label{poincarestr}
\int_X\int_X
|u(x)-u(e^{-t}x+\sqrt{1-e^{-2t}}y)|\,d\gamma(x)d\gamma(y)\leq c_t
|D_\gamma u|(X)
\end{equation}
with
$c_t:=\sqrt{\frac{2}{\pi}}\int_0^t\frac{e^{-s}}{\sqrt{1-e^{-2s}}}\,ds$,
$c_t\sim 2\sqrt{t/\pi}$ as $t\downarrow 0$.
In particular
$$
\int_X|T_tu-u|\,d\gamma\leq c_t|D_\gamma u|(X).
$$
\end{lemma}

Let us now recall the fundamental facts about sets of locally finite perimeter $E$ in $\R^m$. De~Giorgi called
\emph{reduced} boundary of $E$ the set $\mathcal F E$ of points in the support of $|D\chi_E|$ satisfying
$$
\exists\,\,\nu_E(x):=\lim_{r\downarrow 0}\frac{D\chi_E(B_r(x))}{|D\chi_E|(B_r(x))}
\quad\text{and}\quad |\nu_E(x)|=1.
$$
By Besicovitch theorem, $|D\chi_E|$ is concentrated on $\mathcal F E$ and $D\chi_E=\nu_E|D\chi_E|$. 
The main result of \cite{DeG1} are: first, the blown-up sets
\begin{equation}\label{blowneu}
\frac{E-x}{r}
\end{equation}
converge as $r\downarrow 0$ locally in measure, and therefore in $L^1(G_m\Leb{m})$, to the halfspace $S_{\nu_E(x)}$ having $\nu_E$ as inner normal;
second, this information can be used to show that $\mathcal F E$ is \emph{countably $\Haus{m-1}$-rectifiable}, namely there exist
countably many $C^1$ hypersurfaces $\Gamma_i\subset\R^m$ such that 
$$
\Haus{m}\biggl(\mathcal F E\setminus\bigcup_i\Gamma_i\biggr)=0.
$$

In the following results we assume that $(X,\gamma)$ is an $m$-dimensional Gaussian space; if we endow
$X$ with the Cameron-Martin distance $d$, then $(X,\gamma,d)$ is isomorphic to $(\R^m,G_m\Leb{m},\|\cdot\|)$,
$\|\cdot\|$ being the euclidean distance. Under this isomorphism, we have $D_\gamma\chi_E=G_mD\chi_E$ whenever $E$
has finite Gaussian perimeter, so that $|D\chi_E|$ is finite on bounded sets and $E$ has locally finite Euclidean perimeter. 
Since this isomorphism is canonical, we can and shall use it to define $\mathcal F E$ also for sets with finite perimeter 
in $(X,\gamma)$ (although a more intrinsic definition along the lines of the appendix of \cite{ambfig10} could be given).

Having in mind the Ornstein-Uhlenbeck semigroup, the scaling \eqref{blowneu} now becomes
\begin{equation}\label{blowngau}
E_{x,t}:=\frac{E-e^{-t}x}{\sqrt{1-e^{-2t}}},
\end{equation}
so that
$$
T_t\chi_E(x)=\gamma(E_{x,t}).
$$
It corresponds to the scaling \eqref{blowneu} with $r=\sqrt{1-e^{-2t}}\sim\sqrt{2t}$ and with eccentric balls,
whose eccentricity equals $x(e^{-t}-1)$. Since $e^{-t}-1=O(t)=o(r)$, this eccentricity has not effect in the limit
and allows to rewrite, arguing as in \cite[Proposition~3.1]{ambfig10}, the Euclidean statement in Gaussian terms:

\begin{proposition}
Let $(X,\gamma)$ be an $m$-dimensional Gaussian space and $E\subset X$ of finite Gaussian perimeter. Then,
for $|D_\gamma\chi_E|$-a.e. $x\in X$ the rescaled sets $E_{x,t}$ in \eqref{blowngau} converge in $L^2(\gamma)$
to $S_{\nu_E}(x)$.
\end{proposition}

This way, we easily obtain the finite-dimensional version of Theorem~\ref{main}.

As in \cite{ambfig10}, the following lemma (stated with the outer integral in order to avoid measurability
issues) plays a crucial role in the extension to infinite dimensions:

\begin{lemma}\label{lem:3.4}
Let $(X,\gamma)$ be a finite-dimensional Gaussian space, let
$(Y,{\cal F},\mu)$ be a probability space and, for $t>0$ and
$y\in Y$, let $g_{t,y}:X\to [0,1]$ be Borel maps. Assume also that:
\begin{itemize}
\item[(a)]$\{\sigma_y\}_{y\in Y}$ are
positive finite Borel measures in $X$, with
$\int_Y^*\sigma_y(X)\,d\mu(y)$ finite;
\item[(b)] $\sigma_y=G_m\Haus{m-1}\res\Gamma_y$ for $\mu$-a.e. $y$,
with $\Gamma_y$ countably $\Haus{m-1}$-rectifiable.
\end{itemize}
Then
\begin{equation}\label{hino7}
\limsup_{t\downarrow 0}\int_Y^*\int_X T_t
g_{t,y}(x)\,d\sigma_y(x)d\mu(y)\leq\limsup_{t\downarrow 0}
\frac{1}{\sqrt{t}}\int_Y^*\int_X g_{t,y}(x)\,d\gamma(x) d\mu(y).
\end{equation}
\end{lemma}

The proof, given in detail in \cite[Lemma~3.4]{ambfig10}, relies on the heuristic
idea that in an $m$-dimensional Gaussian space $(X,\gamma)$, for the adjoint semigroup $T_t^*$ (i.e. the one
acting on measures) we have
$$
\sqrt{t} T_t^*(G_m\Haus{m-1}\res\Gamma)\leq (1+o(1))\gamma
$$
whenever $\Gamma$ is a $C^1$ hypersurface. This is due to the fact in the case when $\Gamma$ is flat, i.e.
$\Gamma$ is an affine hyperplane, the asymptotic estimate above holds, and that for a non-flat surface only lower order terms
appear. In the flat case, using invariance under
rotation and factorization of the semigroup (see the next section) one is left to the estimate of $\sqrt{t}T_t^*\sigma$
when $X=\R$ and $\sigma$ is a Dirac mass. Then,
considering for instance $\sigma=\delta_0$, a simple computation gives
$$
\sqrt{t}T_t^*(G_m(0)\delta_0)=\frac{1}{2\pi}
\frac{\sqrt{t}}{\sqrt{1-e^{-2t}}}
e^{-|y|^2/(1-e^{-2t})}\Leb{1}\leq \frac{1}{2\sqrt{2\pi}}\gamma+o(1)
\quad\text{as $t\downarrow 0$.}
$$
(See the proof \cite[Lemma~3.4]{ambfig10} for more details.)



\subsection{Factorization of $T_t$ and $D_\gamma u$}

Let us consider the decomposition $X=F\oplus Y$, with $F\subset\tilde{H}$
finite-dimensional. Denoting by $T_t^F$ and $T_t^Y$ the OU
semigroups in $F$ and $Y$ respectively, it is easy to check (for
instance first on products of cylindrical functions on $F$ and $Y$,
and then by linearity and density) that also the action of $T_t$ can
be ``factorized'' in the coordinates $x=(z,y)\in F\times Y$ as
follows:
\begin{equation}\label{factorization}
T_tf(z,y)=T_t^Y\bigl(w\mapsto T_t^F f(\cdot,w)(z)\bigr)(y)
\end{equation}
for any bounded Borel function $f$.

Let us now discuss the factorization properties of $D_\gamma u$.
First of all, we can write $D_\gamma u=\nu_u|D_\gamma u|$ with $\nu_u:X\to H$ Borel
vectorfield satisfying $|\nu_u|_H=1$ $|D_\gamma u|$-a.e. Moreover, given a Borel set $B$, define
$$
B_y:=\left\{z\in F:\ (z,y)\in B\right\},\qquad B_z:=\left\{y\in Y:\
(z,y)\in B\right\}.
$$
The identity
\begin{equation}\label{hino440}
\int_B|\pi_F(\nu_u)|\,d|D_\gamma
u|=\int_Y|D_{\gamma_F}u(\cdot,y)|(B_y)\,d\gamma_Y(y)
\end{equation}
is proved in \cite[Theorem~4.2]{AMP} (see also \cite{AMMP,hin09set}
for analogous results), where $\pi_F:H\to F$ is the orthogonal
projection. Along the similar lines, one can also show the identity
\begin{equation}\label{hino441}
\int_B|\pi_{F^\perp}(\nu_u)|\,d|D_\gamma
u|=\int_F|D_{\gamma_Y}u(z,\cdot)|(B_z)\,d\gamma_F(z)
\end{equation}
with $\pi_F+\pi_{F^\perp}={\rm Id}$.
In the particular case $u=\chi_E$, with the notation
\begin{equation}\label{hino16}
E_y:=\left\{z\in F:\ (z,y)\in E\right\},\qquad E_z:=\left\{y\in Y:\
(z,y)\in E\right\}
\end{equation}
the identities \eqref{hino440} and \eqref{hino441} read respectively
as
\begin{equation}\label{hino40}
\int_B|\pi_F(\nu_E)|\,d|D_\gamma
\chi_E|=\int_Y|D_{\gamma_F}\chi_{E_y}|(B_y)\,d\gamma_Y(y)
\qquad\text{for all $B$ Borel,}
\end{equation}
\begin{equation}\label{hino41}
\int_B|\pi_{F^\perp}(\nu_E)|\,d|D_\gamma
\chi_E|=\int_F|D_{\gamma_Y}\chi_{E_z}|(B_z)\,d\gamma_F(z)
\qquad\text{for all $B$ Borel}
\end{equation}
with $D_\gamma\chi_E=\nu_E|D_\gamma\chi_E|$.

\begin{remark}\label{rtoomuch}{\rm Having in mind \eqref{hino40} and \eqref{hino41},
it is tempting to think that the formula holds for any orthogonal
decomposition of $H$ (so, not only when $F\subset\tilde{H}$), or even when none of the parts if
finite-dimensional. In order to avoid merely technical complications
we shall not treat this issue here because, in this more general
situation, the ``projection maps'' $x\mapsto y$ and $x\mapsto z$ are
no longer continuous. However, the problem can be solved removing sets of
small capacity, see for instance \cite{feydelpra} for a more
detailed discussion.}
\end{remark}

\subsection{Finite-codimension Hausdorff measures}

Following \cite{feydelpra}, we start by introducing pre-Hausdorff
measures which, roughly speaking, play the same role of the
pre-Hausdorff measures $\Haus{n}_\delta$ in the finite-dimensional
theory.

Let $F\subset\tilde{H}$ be a finite-dimensional subspace of dimension $m$,
and for $k\in \mathbb{N}$, $0 \leq k \leq m$, we define (with
the notation of the previous section)
\begin{equation}\label{feydel}
\Haus{\infty-k}_F(B):=\int_Y\int_{B_y}
G_m\,d\Haus{m-k}\,d\gamma_Y(y) \qquad\text{for all $B$ Borel,}
\end{equation}
where $G_m$ is the standard Gaussian density in
$F$ (so that $\Haus{\infty-0}_F=\gamma$). It is proved in
\cite{feydelpra} that $y\mapsto\int_{B_y} G_m\,d\Haus{m-k}$ is
$\gamma_Y$-measurable whenever $B$ is Suslin (so, in particular,
when $B$ is Borel), therefore the integral makes sense. The first
key monotonicity property noticed in \cite{feydelpra}, based on
\cite[2.10.27]{fed}, is
$$
\Haus{\infty-k}_F(B)\leq\Haus{\infty-k}_G(B)\qquad\text{whenever
$F\subset G\subset\tilde{H}$},
$$
provided $\Haus{m-k}$ in \eqref{feydel} is understood as the
\emph{spherical} Hausdorff measure of dimension $m-k$ in $F$. This
naturally leads to the definition
\begin{equation}\label{hino50}
\Haus{\infty-k}(B):=\sup_F\Haus{\infty-k}_F(B),\qquad\text{$B$ Borel,}
\end{equation}
where the supremum runs among all finite-dimensional subspaces $F$
of $\tilde H$. Notice however that, strictly speaking, the measure
defined in \eqref{hino50} does not coincide with the one in
\cite{feydelpra}, since all finite-dimensional subspaces of $H$ are
considered therein. We make the restriction to finite-dimensional
subspaces of $\tilde{H}$ for the reasons explained in
Remark~\ref{rtoomuch}. However, still $\Haus{\infty-k}$ is defined
in a coordinate-free fashion.

These measures have been related for the first time to the perimeter
measure $D_\gamma\chi_E$ in \cite{hin09set}. Hino defined the
$F$-essential boundaries (obtained collecting the essential
boundaries of the finite-dimensional sections $E_y\subset F\times \{y \}$)
\begin{equation}\label{hino17}
\partial_F^*E:=\left\{(z,y):\ z\in\partial^* E_y\right\}
\end{equation}
and noticed another key monotonicity property (see also
\cite[Theorem~5.2]{AMP})
\begin{equation}\label{hino19}
\Haus{\infty-1}_F(\partial^*_FE\setminus\partial^*_GE)=0
\qquad\text{whenever $F\subset G\subset \tilde H$.}
\end{equation}
Then, choosing a sequence ${\cal F}=\{F_1,F_2,\ldots\}$ of
finite-dimensional subspaces of $\tilde{H}$ whose union is dense he
defined
\begin{equation}\label{hino70}
\Haus{\infty-1}_{{\cal F}}:=\sup_n\Haus{\infty-1}_{F_n},\qquad
\partial_{{\cal F}}^*E:=\liminf_{n\to\infty}\partial^*_{F_n}E,
\end{equation}
and showed that
\begin{equation}\label{hino60}
|D_\gamma\chi_E|=\Haus{\infty-1}_{{\cal F}}\res\partial_{{\cal
F}}^*E.
\end{equation}

In order to prove our main result we will follow Hino's procedure, but working with the reduced boundaries
in place of the essential boundaries.

\subsection{Halfspaces}\label{sec:halfspaces}
Let $h\in H$ and $\hat h$ be its corresponding element in $L^2(X,\gamma)$. Then
there exist a linear subspace $X_0\subset X$ such that $\gamma(X\setminus
X_0)=0$ and a representative of $\hat h$ which is 
linear in $X_0$. Indeed, let $h_n\to h$ in $L^2(X,\gamma)$ with $\hat h_n\in X^*$. It is
not restrictive to assume that $\hat h_n \to\hat h$ $\gamma$-a.e. in $X$, so if we define
$$
X_0:=\left\{x\in X:\ \text{$\hat{h}_n(x)$ is a Cauchy sequence}\right\}
$$
we find that $X_0$ is a vector space of full $\gamma$-measure and that the pointwise limit
of $\hat{h}_n$ provides a version of $h$, linear in $X_0$.

Having this fact in mind, it is natural to define halfspaces in the following
way.
\begin{definition}
  Given a unit vector $h\in H$ we shall denote by $S_h$ the halfspace having $h$
  as ``inner normal'', namely
  \begin{equation}\label{def:Sh}
    S_h:=\left\{x\in X:\ \hat{h}(x)>0\right\}.
  \end{equation}
\end{definition}

\begin{proposition} \label{prop:basichalfspace}
For any $S_h$ halfspace it holds $\gamma(S_h)=1/2$, $P(S_h)=\sqrt{1/(2\pi)}$,
and $D\chi_{S_h}=h|D\chi_{S_h}|$.
Furthermore, the 
following implication holds:
$$
\lim_{n\to\infty}|h_n-h|=0\qquad\Longrightarrow\qquad \lim_{n\to\infty}\chi_{S_{h_n}}=\chi_{S_h}.
$$
\end{proposition}
\begin{proof} Let us first show that convergence of $h_n$ to $h$ implies convergence of the corresponding halfspaces.
Since for all $\eps>0$ it holds
$$
\{\hat{h}_n>0\}\setminus\{\hat{h}>0\}\subset
\bigl(\{\hat{h}_n>0\}\setminus\{\hat{h}>-\eps\}\bigr)\cup\{\hat{h}\in (-\eps,0)\}\subset
\{|\hat{h}_n-\hat{h}|>\eps\}\cup\{\hat{h}\in (-\eps,0)\}
$$
and since the convergence of $\hat{h}_n$ to $\hat{h}$ in $L^2(X,\gamma)$ implies
$\gamma(\{|\hat{h}_n-\hat{h}|>\eps\})\to 0$ we obtain
$$
\limsup_{n\to\infty}\gamma(\{\hat{h}_n>0\}\setminus\{\hat{h}>0\})\leq\gamma(\hat{h}^{-1}(-\eps,0)).
$$
Now, since $\hat{h}$ has a standard Gaussian law and $\eps$ is arbitrary it follows that $\gamma(\{\hat{h}_n>0\}\setminus\{\hat{h}>0\})\to 0$.
A similar argument (because the laws of all $\hat{h}_n$ are standard Gaussian) yields $\gamma(\{\hat{h}>0\}\setminus\{\hat{h}_n>0\})\to 0$.

Now, if $\gamma$ is the standard Gaussian in $X=H=\R^n$ and $S_h$ is a halfspace, it is immediate to check that $\gamma(S_h)=1/2$. In addition, since
$D_\gamma \chi_{S_h}=h|D_\gamma \chi_{S_h}|$ and $\hat{h}(x)=\langle h,x\rangle$, we can use $E=S_h$ and $\phi\equiv 1$ in 
the integration by parts formula
$$
\int_{E}\partial_h\phi\, d\gamma+\int_X\phi\,d\langle h, D_\gamma\chi_E\rangle=\int_{E}\hat{h}\,d\gamma
$$
to get $|D_\gamma S_h|(X)=\int_{S_h}\langle h,x\rangle\,dx=\sqrt{1/(2\pi)}$. By a standard cylindrical approximation we obtain that
$\gamma(S_h)=\tfrac12$, $S_h$ has finite perimeter, and $D\chi_{S_h}=h|D\chi_{S_h}|$ in the general case.
\end{proof}

\subsection{Convergence to halfspaces}\label{sec:conv-halfspaces}

In this section we prove Theorem~\ref{main}. 
We consider an increasing family of subspaces $F_n\subset\tilde H$ and, for any $n$, we consider the corresponding decomposition
$x=(x_1,x_2)$ with $x_1\in F_n$ and $x_2\in Y_n$. Denote by $\gamma=\gamma_n\times\gamma_n^\perp$ the corresponding factorization of
$\gamma$. Then, adapting the definition of boundary given in Hino's work \cite{hin09set} (with reduced in
place of essential boundary) we define
$$
\mathcal F_HE:=\liminf_{n\to\infty}B_n\qquad\text{where}\qquad
B_n=\left\{x=(x_1,x_2):\ x_1\in\mathcal FE_{x_2}\right\}
$$
(recall that $E_{x_2}=\left\{x_1\in F_n:\
(x_1,x_2)\in E\right\}$).
We also set $C_n=\cap_{m\geq n}B_m$, so that $C_n\uparrow\mathcal F_HE$ as $n\to\infty$. Recall that
by \eqref{hino440} the measure $\sigma_n:=|\pi_{F_n}(\nu_E)||D_\gamma\chi_E|$ is concentrated on $B_n$, 
because by De~Giorgi's theorem the derivative of finite-dimensional sets of finite perimeter is concentrated on the reduced boundary.
Since $\sigma_n$ is nondecreasing with respect to $n$, $\sigma_n$ is concentrated on all sets $B_m$ with $m\geq n$,
and therefore on $C_n$. It follows that $|D_\gamma\chi_E|=\sup_n\sigma_n$ is concentrated on $\mathcal F_HE$, one of the basic observations in
\cite{hin09set}. 

Let us denote by $\nu_n(x)=\nu_n(x_1,x_2)$ the approximate unit normal to $E_{x_2}^n$ at $x_1$.
Notice that, in this way, $\nu_n$ is
pointwise defined at all points $x\in B_n$ and $D_{\gamma_n}\chi_{E_{x_2}}=\nu_n(x)|D_{\gamma_n}\chi_{E_{x_2}}|$
(again by De~Giorgi's finite-dimensional result). Since the identity (an easy consequence of Fubini's theorem)
$$
\pi_F(D_\gamma\chi_E)=D_{\gamma_n}\chi_{E_{x_2}}\gamma^\perp_n
$$
and the definition of $\nu_n$ give
$$
\pi_{F_n}(\nu_E)|D_\gamma\chi_E|=D_{\gamma_n}\chi_{E_{x_2}}\gamma^\perp_n=\nu_n|D_{\gamma_n} \chi_{E_{x_2}}|\gamma^\perp_n
$$ 
we can use \eqref{hino440} once more to get
$$
\pi_{F_n}(\nu_E)|D_\gamma\chi_E|=\nu_n|\pi_{F_n}(\nu_E)||D_\gamma\chi_E|,
$$
so that $\nu_n=\pi_{F_n}(\nu_E))/|\pi_{F_n}(\nu_E)|$ $\sigma_n$-a.e. in $X$. Since  $\sigma_n\uparrow
|D_\gamma\chi_E|$ as $n\to\infty$, it follows that on each set $C_n$ the function $\nu_m$ is defined for $m\geq n$, and converges
to $\nu_E$  as $m\to\infty$ $|D_\gamma\chi_E|$-a.e. on $C_n$. Then, Proposition \ref{prop:basichalfspace} and the convergence of $\nu_n$ give
\begin{equation}\label{eq:error1}
\lim_{n\to\infty}\int_X\int_X|\chi_{S_{\nu_n}}-\chi_{S_{\nu_E}}|\,d\gamma\,d\sigma_n=0.
\end{equation}
In addition, by the finite-dimensional result of convergence to half spaces, we get
\begin{equation}
\lim_{t\downarrow 0}\int_X\int_{F_n}\left|\chi_{E_{x_2}}(e^{-t}x_1+\sqrt{1-e^{-2t}}x_1')-\chi_{\tilde S_{\nu_n(x)}}(x_1')\right|\,d\gamma_n(x_1')\,d\sigma_n(x)=0,
\end{equation}
where $\tilde S_{\nu_n}$ is the projection of $S_{\nu_n}$ on $F_n$. Now, notice that $S_{\nu_n}=\tilde S_{\nu_n}\times Y_n$, since
$\nu_n\in F$. This observation, in combination with \eqref{eq:error1}, gives that
$$
\limsup_{t\downarrow 0}\int_X\int_X\left|\chi_{E_{x_2}}(e^{-t}x_1+\sqrt{1-e^{-2t}}x_1')-\chi_{S_{\nu_E(x)}}(x')\right|\,d\gamma(x')\,d\sigma_n(x)
$$
is infinitesimal as $n\to\infty$. Therefore to prove \eqref{eq:main} it suffices to show that
\begin{equation}\label{eq:errorefinale}
\limsup_{t\downarrow 0}\int_X\int_X\left|\chi_{E_{x_2}}(e^{-t}x_1+\sqrt{1-e^{-2t}}x_1')-\chi_E(e^{-t}x+\sqrt{1-e^{-2t}}x')\right|
\,d\gamma(x')\,d\sigma_n(x)
\end{equation}
is infinitesimal as $n\to\infty$. 
  
In order to show this last fact, using again $\sigma_n=|D_{\gamma_n}\chi_{E_{x_2}}|\gamma^\perp_n$, 
we can write the expression as 
$$
\limsup_{t\downarrow 0}\int_{Y_n}\int_{F_n}T^{F_n}_t g_t(x_1,x_2)\,d|D_{\gamma_n}\chi_{E_{x_2}}|(x_1)\,d\gamma^\perp_n(x_2)
$$
with $g_t(x_1,x_2):=\int_{Y_n}\left|\chi_E(x_1,x_2')-\chi_E(x_1,e^{-t}x_2+\sqrt{1-e^{-2t}}x_2')\right|\,d\gamma_n^\perp(x_2')$. 
As in \cite{ambfig10} we now use Lemma~\ref{lem:3.4} and the rectifiability of the measures $|D_{\gamma_n}\chi_{E_{x_2}}|$
to bound the limsup above by
\begin{equation}\label{eq:errorefinale1}
\limsup_{t\downarrow 0}\int_{Y_n}\int_{F_n}\frac{g_t(x_1,x_2)}{\sqrt{t}}\,d\gamma_n(x_1)\,d\gamma^\perp_n(x_2).
\end{equation}
Now we integrate w.r.t. $\gamma_n$ the inequality (ensured by \eqref{poincarestr})
$$
\int_{Y_n}g_t(x_1,x_2)\,d\gamma_n^\perp(x_2)\leq c\sqrt{t}|D_{\gamma^\perp_n}\chi_{E_{x_1}}|(Y_n),
$$
valid for all $x_1$ such that $E_{x_1}$ has finite perimeter in $(Y_n,\gamma^\perp_n)$,
to bound the $\limsup$ in \eqref{eq:errorefinale1} by
$$c\int_{F_n}|D_{\gamma^\perp_n}\chi_{E_{x_1}}|(Y_n)\,d\gamma_n(x_1)=c\int_X|\pi_{F_n}^\perp(\nu_E)|\,d|D_\gamma\chi_E|.$$
Since $|\pi_{F_n}^\perp\nu_E|\downarrow 0$ as $n\to\infty$, this concludes the proof.

\bibliographystyle{plain}

\end{document}